\documentclass{amsart}
\usepackage{amssymb}
\usepackage{amsmath}
\usepackage{amsfonts}

\newtheorem{theorem}{Theorem}
\theoremstyle{plain}

\newtheorem{corollary}{Corollary}

\newtheorem{definition}{Definition}

\newtheorem{lemma}{Lemma}

\newtheorem{proposition}{Proposition}
\newtheorem{remark}{Remark}

\numberwithin{equation}{section}
\input{tcilatex}

\begin{document}
\title[Integral inequalities]{On new general integral inequalities for $s$%
-convex functions}
\author{\.{I}mdat \.{I}\c{s}can$^{\blacktriangledown }$}
\address{$^{\blacktriangledown }$Department of Mathematics, Faculty of Arts
and Sciences, Giresun University, 28100, Giresun, Turkey.}
\email{imdat.iscan@giresun.edu.tr, imdati@yahoo.com}
\author{Erhan SET$^{\clubsuit }$}
\address{$^{\clubsuit }$Department of Mathematics, Faculty of Arts and
Sciences, Ordu University, 52200, Ordu, Turkey}
\email{erhanset@yahoo.com}
\author{M. Emin \"{O}zdemir$^{\blacksquare }$}
\address{$^{\blacksquare }$Atat\"{u}rk University, K.K. Education Faculty,
Department of Mathematics, 25240, Campus, Erzurum, Turkey}
\email{emos@atauni.edu.tr}
\date{}

\begin{abstract}
In this paper, the authors establish some new estimates for the remainder
term of the midpoint, trapezoid, and Simpson formula using functions whose
derivatives in absolute value at certain power are $s$-convex. Some
applications to special means of real numbers are provided as well.
\end{abstract}

\maketitle

\section{Introduction}

Let $f:I\subseteq \mathbb{R\rightarrow R}$ be a convex function defined on
the interval $I$ of real numbers and $a,b\in I$ with $a<b$. The following
inequality%
\begin{equation}
f\left( \frac{a+b}{2}\right) \leq \frac{1}{b-a}\dint\limits_{a}^{b}f(x)dx%
\leq \frac{f(a)+f(b)}{2}  \label{1-1}
\end{equation}%
holds. This double inequality is known in the literature as Hermite-Hadamard
integral inequality for convex functions. See (\cite{ADK11},\cite{DF99},\cite%
{I12}-\cite{P10},\cite{SSOH}) for the results of the generalization,
improvement and extention of the famous integral inequality (\ref{1-1}).

In 1978, Breckner introduced $s$-convex functions as a generalization of
convex functions as follows \cite{B78}:

\begin{definition}
Let $s\in (0,1]$ be a fixed real number. A function $f:[0,\infty
)\rightarrow \lbrack 0,\infty )$ is said to be $s-$convex (in the second
sense),or that $f$ belongs to the class $K_{s}^{2}$, if \ 
\begin{equation*}
f\left( \alpha x+(1-\alpha )y\right) \leq \alpha ^{s}f(x)+(1-\alpha )^{s}f(y)
\end{equation*}%
for all $x,y\in \lbrack 0,\infty )$ and $\alpha \in \lbrack 0,1]$.
\end{definition}

Of course, $s$-convexity means just convexity when $s=1$. For other recent
results concerning s-convex functions see \cite{ADD09}-\cite{set4}.

The following inequality is well known in the literature as Simpson's
inequality:

Let $f:\left[ a,b\right] \mathbb{\rightarrow R}$ be a four times
continuously differentiable mapping on $\left( a,b\right) $ and $\left\Vert
f^{(4)}\right\Vert _{\infty }=\underset{x\in \left( a,b\right) }{\sup }%
\left\vert f^{(4)}(x)\right\vert <\infty .$ Then the following inequality
holds:%
\begin{equation*}
\left\vert \frac{1}{3}\left[ \frac{f(a)+f(b)}{2}+2f\left( \frac{a+b}{2}%
\right) \right] -\frac{1}{b-a}\dint\limits_{a}^{b}f(x)dx\right\vert \leq 
\frac{1}{2880}\left\Vert f^{(4)}\right\Vert _{\infty }\left( b-a\right) ^{4}.
\end{equation*}%
\ \qquad In recent years many authors have studied error estimations for
Simpson's inequality. For refinements, counterparts, generalizations of the
Simpson's inequality and new Simpson's type inequalities, see \cite%
{ADD09,I12,I13b,I13d,SA11,SSO10,set4}.

In \cite{DF99}, Dragomir and Fitzpatrick proved a variant of
Hermite--Hadamard inequality which holds for the $s$-convex functions.

\begin{theorem}
Suppose that $f:\left[ 0,\infty \right) \rightarrow \left[ 0,\infty \right) $
is an s-convex function in the second sense, where $s\in (0,1]$ and let $%
a,b\in \left[ 0,\infty \right) $, $a<b$. If $f\in L\left[ a,b\right] $, then
the following inequalities hold%
\begin{equation}
2^{s-1}f\left( \frac{a+b}{2}\right) \leq \frac{1}{b-a}\dint%
\limits_{a}^{b}f(x)dx\leq \frac{f(a)+f(b)}{s+1}  \label{1-2}
\end{equation}%
the constant $k=\frac{1}{s+1}$ is the best possible in the second inequality
in (\ref{1-2}). The above inequalities are sharp.
\end{theorem}

In \cite{I12}, Iscan obtained\ a new generalization of some integral
inequalities for differentiable convex mapping which are connected Simpson
and Hadamard type inequalities, and he used the following lemma to prove
this.

\begin{lemma}
\label{2.1}Let $f:I\subseteq \mathbb{R\rightarrow R}$ be a differentiable
mapping on $I^{\circ }$ such that $f^{\prime }\in L[a,b]$, where $a,b\in I$
with $a<b$ and $\alpha ,\lambda \in \left[ 0,1\right] $. Then the following
equality holds:%
\begin{eqnarray*}
&&\lambda \left( \alpha f(a)+\left( 1-\alpha \right) f(b)\right) +\left(
1-\lambda \right) f(\alpha a+\left( 1-\alpha \right) b)-\frac{1}{b-a}%
\dint\limits_{a}^{b}f(x)dx \\
&=&\left( b-a\right) \left[ \dint\limits_{0}^{1-\alpha }\left( t-\alpha
\lambda \right) f^{\prime }\left( tb+(1-t)a\right) dt\right. \\
&&\left. +\dint\limits_{1-\alpha }^{1}\left( t-1+\lambda \left( 1-\alpha
\right) \right) f^{\prime }\left( tb+(1-t)a\right) dt\right] .
\end{eqnarray*}
\end{lemma}

The main inequality in \cite{I12}, pointed out, is as follows.

\begin{theorem}
Let $f:I\subseteq \mathbb{R\rightarrow R}$ be a differentiable mapping on $%
I^{\circ }$ such that $f^{\prime }\in L[a,b]$, where $a,b\in I^{\circ }$
with $a<b$ and $\alpha ,\lambda \in \left[ 0,1\right] $. If $\left\vert
f^{\prime }\right\vert ^{q}$ is convex on $[a,b]$, $q\geq 1,$ then the
following inequality holds:%
\begin{equation*}
\left\vert \lambda \left( \alpha f(a)+\left( 1-\alpha \right) f(b)\right)
+\left( 1-\lambda \right) f(\alpha a+\left( 1-\alpha \right) b)-\frac{1}{b-a}%
\dint\limits_{a}^{b}f(x)dx\right\vert
\end{equation*}%
\begin{eqnarray}
&&  \label{1-3} \\
&\leq &\left\{ 
\begin{array}{cc}
\begin{array}{c}
\left( b-a\right) \left\{ \gamma _{2}^{1-\frac{1}{q}}\left( \mu
_{1}\left\vert f^{\prime }(b)\right\vert ^{q}+\mu _{2}\left\vert f^{\prime
}(a)\right\vert ^{q}\right) ^{\frac{1}{q}}\right. \\ 
\left. +\upsilon _{2}^{1-\frac{1}{q}}\left( \eta _{3}\left\vert f^{\prime
}(b)\right\vert ^{q}+\eta _{4}\left\vert f^{\prime }(a)\right\vert
^{q}\right) ^{\frac{1}{q}}\right\} ,%
\end{array}
& \alpha \lambda \leq 1-\alpha \leq 1-\lambda \left( 1-\alpha \right) \\ 
\begin{array}{c}
\left( b-a\right) \left\{ \gamma _{2}^{1-\frac{1}{q}}\left( \mu
_{1}\left\vert f^{\prime }(b)\right\vert ^{q}+\mu _{2}\left\vert f^{\prime
}(a)\right\vert ^{q}\right) ^{\frac{1}{q}}\right. \\ 
\left. +\upsilon _{1}^{1-\frac{1}{q}}\left( \eta _{1}\left\vert f^{\prime
}(b)\right\vert ^{q}+\eta _{2}\left\vert f^{\prime }(a)\right\vert
^{q}\right) ^{\frac{1}{q}}\right\} ,%
\end{array}
& \alpha \lambda \leq 1-\lambda \left( 1-\alpha \right) \leq 1-\alpha \\ 
\begin{array}{c}
\left( b-a\right) \left\{ \gamma _{1}^{1-\frac{1}{q}}\left( \mu
_{3}\left\vert f^{\prime }(b)\right\vert ^{q}+\mu _{4}\left\vert f^{\prime
}(a)\right\vert ^{q}\right) ^{\frac{1}{q}}\right. \\ 
\left. +\upsilon _{2}^{1-\frac{1}{q}}\left( \eta _{3}\left\vert f^{\prime
}(b)\right\vert ^{q}+\eta _{4}\left\vert f^{\prime }(a)\right\vert
^{q}\right) ^{\frac{1}{q}}\right\} ,%
\end{array}
& 1-\alpha \leq \alpha \lambda \leq 1-\lambda \left( 1-\alpha \right)%
\end{array}%
\right.  \notag
\end{eqnarray}%
where 
\begin{equation*}
\gamma _{1}=\left( 1-\alpha \right) \left[ \alpha \lambda -\frac{\left(
1-\alpha \right) }{2}\right] ,\ \gamma _{2}=\left( \alpha \lambda \right)
^{2}-\gamma _{1}\ ,
\end{equation*}%
\begin{eqnarray*}
\upsilon _{1} &=&\frac{1-\left( 1-\alpha \right) ^{2}}{2}-\alpha \left[
1-\lambda \left( 1-\alpha \right) \right] , \\
\upsilon _{2} &=&\frac{1+\left( 1-\alpha \right) ^{2}}{2}-\left( \lambda
+1\right) \left( 1-\alpha \right) \left[ 1-\lambda \left( 1-\alpha \right) %
\right] ,
\end{eqnarray*}%
\begin{eqnarray*}
\mu _{1} &=&\frac{\left( \alpha \lambda \right) ^{3}+\left( 1-\alpha \right)
^{3}}{3}-\alpha \lambda \frac{\left( 1-\alpha \right) ^{2}}{2},\  \\
\mu _{2} &=&\frac{1+\alpha ^{3}+\left( 1-\alpha \lambda \right) ^{3}}{3}-%
\frac{\left( 1-\alpha \lambda \right) }{2}\left( 1+\alpha ^{2}\right) , \\
\mu _{3} &=&\alpha \lambda \frac{\left( 1-\alpha \right) ^{2}}{2}-\frac{%
\left( 1-\alpha \right) ^{3}}{3}, \\
\mu _{4} &=&\frac{\left( \alpha \lambda -1\right) \left( 1-\alpha
^{2}\right) }{2}+\frac{1-\alpha ^{3}}{3},
\end{eqnarray*}%
\begin{eqnarray*}
\eta _{1} &=&\frac{1-\left( 1-\alpha \right) ^{3}}{3}-\frac{\left[ 1-\lambda
\left( 1-\alpha \right) \right] }{2}\alpha \left( 2-\alpha \right) ,\  \\
\eta _{2} &=&\frac{\lambda \left( 1-\alpha \right) \alpha ^{2}}{2}-\frac{%
\alpha ^{3}}{3},
\end{eqnarray*}%
\begin{eqnarray*}
\eta _{3} &=&\frac{\left[ 1-\lambda \left( 1-\alpha \right) \right] ^{3}}{3}-%
\frac{\left[ 1-\lambda \left( 1-\alpha \right) \right] }{2}\left( 1+\left(
1-\alpha \right) ^{2}\right) +\frac{1+\left( 1-\alpha \right) ^{3}}{3}, \\
\eta _{4} &=&\frac{\left[ \lambda \left( 1-\alpha \right) \right] ^{3}}{3}-%
\frac{\lambda \left( 1-\alpha \right) \alpha ^{2}}{2}+\frac{\alpha ^{3}}{3}.
\end{eqnarray*}
\end{theorem}

In \cite{ADK11} Alomari et al. obtained the following inequalities of the
left-hand side of Hermite-Hadamard's inequality for $s$-convex mappings.

\begin{theorem}
Let $f:I\subseteq \lbrack 0,\infty )\rightarrow \mathbb{R}$ be a
differentiable mapping on $I^{\circ }$, such that $f^{\prime }\in L[a,b]$,
where $a,b\in I$ with $a<b$. If $|f^{\prime }|^{q},\ q\geq 1,$ is $s$-convex
on $[a,b]$, for some fixed $s\in (0,1]$, then the following inequality holds:%
\begin{eqnarray}
&&\left\vert f\left( \frac{a+b}{2}\right) -\frac{1}{b-a}\dint%
\limits_{a}^{b}f(x)dx\right\vert  \notag \\
&\leq &\frac{b-a}{8}\left( \frac{2}{(s+1)(s+2)}\right) ^{\frac{1}{q}}\left[
\left\{ \left( 2^{1-s}+1\right) \left\vert f^{\prime }(b)\right\vert
^{q}+2^{1-s}\left\vert f^{\prime }(a)\right\vert ^{q}\right\} ^{\frac{1}{q}%
}\right.  \notag \\
&&\left. +\left\{ \left( 2^{1-s}+1\right) \left\vert f^{\prime
}(a)\right\vert ^{q}+2^{1-s}\left\vert f^{\prime }(b)\right\vert
^{q}\right\} ^{\frac{1}{q}}\right] .  \label{1-4}
\end{eqnarray}
\end{theorem}

\begin{theorem}
Let $f:I\subseteq \lbrack 0,\infty )\rightarrow \mathbb{R}$ be a
differentiable mapping on $I^{\circ }$, such that $f^{\prime }\in L[a,b]$,
where $a,b\in I$ with $a<b$. If $|f^{\prime }|^{\frac{p}{p-1}},\ p>1,$ is $s$%
-convex on $[a,b]$, for some fixed $s\in (0,1]$, then the following
inequality holds:%
\begin{eqnarray}
\left\vert f\left( \frac{a+b}{2}\right) -\frac{1}{b-a}\dint%
\limits_{a}^{b}f(x)dx\right\vert &\leq &\left( \frac{b-a}{4}\right) \left( 
\frac{1}{p+1}\right) ^{\frac{1}{p}}\left( \frac{1}{s+1}\right) ^{\frac{2}{q}}
\notag \\
&&\times \left[ \left( \left( 2^{1-s}+s+1\right) \left\vert f^{\prime
}\left( a\right) \right\vert ^{q}+2^{1-s}\left\vert f^{\prime }\left(
b\right) \right\vert ^{q}\right) ^{\frac{1}{q}}\right.  \notag \\
&&+\left. \left( \left( 2^{1-s}+s+1\right) \left\vert f^{\prime }\left(
b\right) \right\vert ^{q}+2^{1-s}\left\vert f^{\prime }\left( a\right)
\right\vert ^{q}\right) ^{\frac{1}{q}}\right] ,  \label{1-4a}
\end{eqnarray}%
where $p$ is the conjugate of $q$, $q=p/(p-1).$
\end{theorem}

In \cite{SSO10}, Sarikaya et al. obtained a new upper bound for the
right-hand side of Simpson's inequality for $s-$convex mapping as follows:

\begin{theorem}
Let $f:I\subseteq \lbrack 0,\infty )\rightarrow \mathbb{R}$ be a
differentiable mapping on $I^{\circ }$, such that $f^{\prime }\in L[a,b]$,
where $a,b\in I^{\circ }$ with $a<b$. If $|f^{\prime }|^{q},\ $ is $s$%
-convex on $[a,b]$, for some fixed $s\in (0,1]$ and $q>1,$ then the
following inequality holds:%
\begin{eqnarray}
&&\left\vert \frac{1}{6}\left[ f(a)+4f\left( \frac{a+b}{2}\right) +f(b)%
\right] -\frac{1}{b-a}\dint\limits_{a}^{b}f(x)dx\right\vert \leq \frac{b-a}{%
12}\left( \frac{1+2^{p+1}}{3\left( p+1\right) }\right) ^{\frac{1}{p}}
\label{1-5} \\
&&\times \left\{ \left( \frac{\left\vert f^{\prime }\left( \frac{a+b}{2}%
\right) \right\vert ^{q}+\left\vert f^{\prime }\left( a\right) \right\vert
^{q}}{s+1}\right) ^{\frac{1}{q}}+\left( \frac{\left\vert f^{\prime }\left( 
\frac{a+b}{2}\right) \right\vert ^{q}+\left\vert f^{\prime }\left( b\right)
\right\vert ^{q}}{s+1}\right) ^{\frac{1}{q}}\right\} ,  \notag
\end{eqnarray}%
where $\frac{1}{p}+\frac{1}{q}=1.$
\end{theorem}

In \cite{KBOP07}, Kirmaci et al. proved the following trapezoid inequality:

\begin{theorem}
Let $f:I\subseteq \lbrack 0,\infty )\rightarrow \mathbb{R}$ be a
differentiable mapping on $I^{\circ }$, such that $f^{\prime }\in L[a,b]$,
where $a,b\in I^{\circ }$, $a<b$. If $|f^{\prime }|^{q},\ $ is $s$-convex on 
$[a,b]$, for some fixed $s\in (0,1)$ and $q>1,$ then%
\begin{eqnarray}
&&\left\vert \frac{f\left( a\right) +f\left( b\right) }{2}-\frac{1}{b-a}%
\dint\limits_{a}^{b}f(x)dx\right\vert \leq \frac{b-a}{2}\left( \frac{q-1}{%
2\left( 2q-1\right) }\right) ^{\frac{q-1}{q}}\left( \frac{1}{s+1}\right) ^{%
\frac{1}{q}}  \label{1-6} \\
&&\times \left\{ \left( \left\vert f^{\prime }\left( \frac{a+b}{2}\right)
\right\vert ^{q}+\left\vert f^{\prime }\left( a\right) \right\vert
^{q}\right) ^{\frac{1}{q}}+\left( \left\vert f^{\prime }\left( \frac{a+b}{2}%
\right) \right\vert ^{q}+\left\vert f^{\prime }\left( b\right) \right\vert
^{q}\right) ^{\frac{1}{q}}\right\} .  \notag
\end{eqnarray}
\end{theorem}

\section{Main results}

Let $f:I\subseteq 
\mathbb{R}
\rightarrow 
\mathbb{R}
$ be a differentiable function on $I^{\circ }$, the interior of $I$,
throughout this section we will take%
\begin{eqnarray*}
&&I_{f}\left( \lambda ,\alpha ,a,b\right) \\
&=&\lambda \left( \alpha f(a)+\left( 1-\alpha \right) f(b)\right) +\left(
1-\lambda \right) f(\alpha a+\left( 1-\alpha \right) b)-\frac{1}{b-a}%
\dint\limits_{a}^{b}f(x)dx
\end{eqnarray*}%
where $a,b\in I^{\circ }$ with $a<b$ and $\alpha ,\lambda \in \left[ 0,1%
\right] $.

\begin{theorem}
\label{2.2}Let $f:I\subseteq \mathbb{R\rightarrow R}$ be a differentiable
mapping on $I^{\circ }$ such that $f^{\prime }\in L[a,b]$, where $a,b\in
I^{\circ }$ with $a<b$ and $\alpha ,\lambda \in \left[ 0,1\right] $. If $%
\left\vert f^{\prime }\right\vert ^{q}$ is $s$-convex on $[a,b]$, for some
fixed $s\in (0,1]$ \ and $q\geq 1,$ then

(i) for $\alpha \lambda \leq 1-\alpha \leq 1-\lambda \left( 1-\alpha \right) 
$ we have%
\begin{eqnarray*}
\left\vert I_{f}\left( \lambda ,\alpha ,a,b\right) \right\vert &\leq &\left(
b-a\right) \left[ \gamma _{2}^{1-\frac{1}{q}}(\alpha ,\lambda )\left(
c_{1}(\alpha ,\lambda ,s)\left\vert f^{\prime }(b)\right\vert
^{q}+c_{2}(\alpha ,\lambda ,s)\left\vert f^{\prime }(a)\right\vert
^{q}\right) ^{\frac{1}{q}}\right. \\
&&+\left. \gamma _{2}^{1-\frac{1}{q}}(1-\alpha ,\lambda )\left(
c_{2}(1-\alpha ,\lambda ,s)\left\vert f^{\prime }(b)\right\vert
^{q}+c_{1}(1-\alpha ,\lambda ,s)\left\vert f^{\prime }(a)\right\vert
^{q}\right) ^{\frac{1}{q}}\right] ,
\end{eqnarray*}

(ii) for $\alpha \lambda \leq 1-\lambda \left( 1-\alpha \right) \leq
1-\alpha $ we have%
\begin{eqnarray*}
\left\vert I_{f}\left( \lambda ,\alpha ,a,b\right) \right\vert &\leq &\left(
b-a\right) \left[ \gamma _{2}^{1-\frac{1}{q}}(\alpha ,\lambda )\left(
c_{1}(\alpha ,\lambda ,s)\left\vert f^{\prime }(b)\right\vert
^{q}+c_{2}(\alpha ,\lambda ,s)\left\vert f^{\prime }(a)\right\vert
^{q}\right) ^{\frac{1}{q}}\right. \\
&&+\left. \gamma _{1}^{1-\frac{1}{q}}(1-\alpha ,\lambda )\left(
c_{4}(1-\alpha ,\lambda ,s)\left\vert f^{\prime }(b)\right\vert
^{q}+c_{3}(1-\alpha ,\lambda ,s)\left\vert f^{\prime }(a)\right\vert
^{q}\right) ^{\frac{1}{q}}\right] ,
\end{eqnarray*}

(iii) for $1-\alpha \leq \alpha \lambda \leq 1-\lambda \left( 1-\alpha
\right) $ we have%
\begin{eqnarray*}
\left\vert I_{f}\left( \lambda ,\alpha ,a,b\right) \right\vert &\leq &\left(
b-a\right) \left[ \gamma _{1}^{1-\frac{1}{q}}(\alpha ,\lambda )\left(
c_{3}(\alpha ,\lambda ,s)\left\vert f^{\prime }(b)\right\vert
^{q}+c_{4}(\alpha ,\lambda ,s)\left\vert f^{\prime }(a)\right\vert
^{q}\right) ^{\frac{1}{q}}\right. \\
&&+\left. \gamma _{2}^{1-\frac{1}{q}}(1-\alpha ,\lambda )\left(
c_{2}(1-\alpha ,\lambda ,s)\left\vert f^{\prime }(b)\right\vert
^{q}+c_{1}(1-\alpha ,\lambda ,s)\left\vert f^{\prime }(a)\right\vert
^{q}\right) ^{\frac{1}{q}}\right]
\end{eqnarray*}
where 
\begin{eqnarray*}
\gamma _{1}(\alpha ,\lambda ) &=&\left( 1-\alpha \right) \left[ \alpha
\lambda -\frac{\left( 1-\alpha \right) }{2}\right] , \\
\ \gamma _{2}(\alpha ,\lambda ) &=&\left( \alpha \lambda \right) ^{2}-\gamma
_{1}(\alpha ,\lambda )\ ,
\end{eqnarray*}%
\begin{eqnarray*}
c_{1}(\alpha ,\lambda ,s) &=&\left( \alpha \lambda \right) ^{s+2}\frac{2}{%
\left( s+1\right) \left( s+2\right) }-\left( \alpha \lambda \right) \frac{%
\left( 1-\alpha \right) ^{s+1}}{s+1}+\frac{\left( 1-\alpha \right) ^{s+2}}{%
s+2}, \\
c_{2}(\alpha ,\lambda ,s) &=&\left( 1-\alpha \lambda \right) ^{s+2}\frac{2}{%
\left( s+1\right) \left( s+2\right) }-\frac{\left( 1-\alpha \lambda \right)
\left( 1+\alpha ^{s+1}\right) }{s+1}+\frac{1+\alpha ^{s+2}}{s+2}, \\
c_{3}(\alpha ,\lambda ,s) &=&\left( \alpha \lambda \right) \frac{\left(
1-\alpha \right) ^{s+1}}{s+1}-\frac{\left( 1-\alpha \right) ^{s+2}}{s+2}, \\
c_{4}(\alpha ,\lambda ,s) &=&\frac{\left( \alpha \lambda -1\right) \left(
1-\alpha ^{s+1}\right) }{s+1}+\frac{1-\alpha ^{s+2}}{s+2}.
\end{eqnarray*}
\end{theorem}

\begin{proof}
Suppose that $q\geq 1.$ From Lemma \ref{2.1} and using the well known power
mean inequality, we have%
\begin{eqnarray*}
&&\left\vert I_{f}\left( \lambda ,\alpha ,a,b\right) \right\vert \\
&\leq &\left( b-a\right) \left[ \dint\limits_{0}^{1-\alpha }\left\vert
t-\alpha \lambda \right\vert \left\vert f^{\prime }\left( tb+(1-t)a\right)
\right\vert dt+\dint\limits_{1-\alpha }^{1}\left\vert t-1+\lambda \left(
1-\alpha \right) \right\vert \left\vert f^{\prime }\left( tb+(1-t)a\right)
\right\vert dt\right] \\
&\leq &\left( b-a\right) \left\{ \left( \dint\limits_{0}^{1-\alpha
}\left\vert t-\alpha \lambda \right\vert dt\right) ^{1-\frac{1}{q}}\left(
\dint\limits_{0}^{1-\alpha }\left\vert t-\alpha \lambda \right\vert
\left\vert f^{\prime }\left( tb+(1-t)a\right) \right\vert ^{q}dt\right) ^{%
\frac{1}{q}}\right.
\end{eqnarray*}%
\begin{equation}
\left. +\left( \dint\limits_{1-\alpha }^{1}\left\vert t-1+\lambda \left(
1-\alpha \right) \right\vert dt\right) ^{1-\frac{1}{q}}\left(
\dint\limits_{1-\alpha }^{1}\left\vert t-1+\lambda \left( 1-\alpha \right)
\right\vert \left\vert f^{\prime }\left( tb+(1-t)a\right) \right\vert
^{q}dt\right) ^{\frac{1}{q}}\right\}  \label{2-3}
\end{equation}%
Consider 
\begin{equation*}
I_{1}=\dint\limits_{0}^{1-\alpha }\left\vert t-\alpha \lambda \right\vert
\left\vert f^{\prime }\left( tb+(1-t)a\right) \right\vert ^{q}dt,\ \
I_{2}=\dint\limits_{1-\alpha }^{1}\left\vert t-1+\lambda \left( 1-\alpha
\right) \right\vert \left\vert f^{\prime }\left( tb+(1-t)a\right)
\right\vert ^{q}dt
\end{equation*}

Since $\left\vert f^{\prime }\right\vert ^{q}$ is $s$-convex on $[a,b],$%
\begin{equation}
I_{1}\leq \left\vert f^{\prime }(b)\right\vert
^{q}\dint\limits_{0}^{1-\alpha }\left\vert t-\alpha \lambda \right\vert
t^{s}dt+\left\vert f^{\prime }(a)\right\vert ^{q}\dint\limits_{0}^{1-\alpha
}\left\vert t-\alpha \lambda \right\vert (1-t)^{s}dt.  \label{2-4}
\end{equation}%
Similarly%
\begin{equation}
I_{2}\leq \left\vert f^{\prime }(b)\right\vert ^{q}\dint\limits_{1-\alpha
}^{1}\left\vert t-1+\lambda \left( 1-\alpha \right) \right\vert
t^{s}dt+\left\vert f^{\prime }(a)\right\vert ^{q}\dint\limits_{1-\alpha
}^{1}\left\vert t-1+\lambda \left( 1-\alpha \right) \right\vert (1-t)^{s}dt.
\label{2-5}
\end{equation}%
Additionally, by simple computation%
\begin{equation}
\dint\limits_{0}^{1-\alpha }\left\vert t-\alpha \lambda \right\vert
dt=\left\{ 
\begin{array}{cc}
\gamma _{2}(\alpha ,\lambda ), & \alpha \lambda \leq 1-\alpha \\ 
\gamma _{1}(\alpha ,\lambda ), & \alpha \lambda \geq 1-\alpha%
\end{array}%
\right. ,  \label{2-6}
\end{equation}%
\begin{equation*}
\gamma _{1}(\alpha ,\lambda )=\left( 1-\alpha \right) \left[ \alpha \lambda -%
\frac{\left( 1-\alpha \right) }{2}\right] ,\ \gamma _{2}(\alpha ,\lambda
)=\left( \alpha \lambda \right) ^{2}-\gamma _{1}(\alpha ,\lambda )\ ,
\end{equation*}%
\begin{equation*}
\dint\limits_{1-\alpha }^{1}\left\vert t-1+\lambda \left( 1-\alpha \right)
\right\vert dt=\dint\limits_{0}^{\alpha }\left\vert t-\left( 1-\alpha
\right) \lambda \right\vert dt
\end{equation*}%
\begin{equation*}
=\left\{ 
\begin{array}{cc}
\gamma _{1}(1-\alpha ,\lambda ), & 1-\lambda \left( 1-\alpha \right) \leq
1-\alpha \\ 
\gamma _{2}(1-\alpha ,\lambda ), & 1-\lambda \left( 1-\alpha \right) \geq
1-\alpha%
\end{array}%
\right. ,
\end{equation*}%
\begin{equation*}
\dint\limits_{0}^{1-\alpha }\left\vert t-\alpha \lambda \right\vert
t^{s}dt=\left\{ 
\begin{array}{cc}
c_{1}(\alpha ,\lambda ,s), & \alpha \lambda \leq 1-\alpha \\ 
c_{3}(\alpha ,\lambda ,s), & \alpha \lambda \geq 1-\alpha%
\end{array}%
\right.
\end{equation*}%
\begin{equation*}
\dint\limits_{0}^{1-\alpha }\left\vert t-\alpha \lambda \right\vert
(1-t)^{s}dt=\left\{ 
\begin{array}{cc}
c_{2}(\alpha ,\lambda ,s), & \alpha \lambda \leq 1-\alpha \\ 
c_{4}(\alpha ,\lambda ,s), & \alpha \lambda \geq 1-\alpha%
\end{array}%
\right.
\end{equation*}%
\begin{equation*}
\dint\limits_{1-\alpha }^{1}\left\vert t-1+\lambda \left( 1-\alpha \right)
\right\vert t^{s}=\dint\limits_{0}^{\alpha }\left\vert t-\left( 1-\alpha
\right) \lambda \right\vert (1-t)^{s}dt
\end{equation*}%
\begin{equation*}
=\left\{ 
\begin{array}{cc}
c_{4}(1-\alpha ,\lambda ,s), & 1-\lambda \left( 1-\alpha \right) \leq
1-\alpha \\ 
c_{2}(1-\alpha ,\lambda ,s), & 1-\lambda \left( 1-\alpha \right) \geq
1-\alpha%
\end{array}%
\right.
\end{equation*}%
\begin{equation*}
\dint\limits_{1-\alpha }^{1}\left\vert t-1+\lambda \left( 1-\alpha \right)
\right\vert (1-t)^{s}dt=\dint\limits_{0}^{\alpha }\left\vert t-\left(
1-\alpha \right) \lambda \right\vert t^{s}dt
\end{equation*}%
\begin{equation}
=\left\{ 
\begin{array}{cc}
c_{3}(1-\alpha ,\lambda ,s), & 1-\lambda \left( 1-\alpha \right) \leq
1-\alpha \\ 
c_{1}(1-\alpha ,\lambda ,s), & 1-\lambda \left( 1-\alpha \right) \geq
1-\alpha%
\end{array}%
\right. .  \label{2-7}
\end{equation}%
Thus, using (\ref{2-4})-(\ref{2-7}) in (\ref{2-3}), we obtain desired
results. This completes the proof.
\end{proof}

\begin{corollary}
Under the assumptions of Theorem \ref{2.2} with $q=1,$

(i) if $\alpha \lambda \leq 1-\alpha \leq 1-\lambda \left( 1-\alpha \right) $%
, then we have%
\begin{eqnarray*}
\left\vert I_{f}\left( \lambda ,\alpha ,a,b\right) \right\vert &\leq &\left(
b-a\right) \left[ \left( c_{1}(\alpha ,\lambda ,s)+c_{2}(1-\alpha ,\lambda
,s)\right) \left\vert f^{\prime }(b)\right\vert \right. \\
&&+\left. \left( c_{2}(\alpha ,\lambda ,s)+c_{1}(1-\alpha ,\lambda
,s)\right) \left\vert f^{\prime }(a)\right\vert \right] ,
\end{eqnarray*}

(ii) if $\alpha \lambda \leq 1-\lambda \left( 1-\alpha \right) \leq 1-\alpha 
$, then we have%
\begin{eqnarray*}
\left\vert I_{f}\left( \lambda ,\alpha ,a,b\right) \right\vert &\leq &\left(
b-a\right) \left[ \left( c_{1}(\alpha ,\lambda ,s)+c_{4}(1-\alpha ,\lambda
,s)\right) \left\vert f^{\prime }(b)\right\vert \right. \\
&&+\left. \left( c_{2}(\alpha ,\lambda ,s)+c_{3}(1-\alpha ,\lambda
,s)\right) \left\vert f^{\prime }(a)\right\vert \right] ,
\end{eqnarray*}%
(iii) if $1-\alpha \leq \alpha \lambda \leq 1-\lambda \left( 1-\alpha
\right) $, then we have%
\begin{eqnarray*}
\left\vert I_{f}\left( \lambda ,\alpha ,a,b\right) \right\vert &\leq &\left(
b-a\right) \left[ \left( c_{3}(\alpha ,\lambda ,s)+c_{2}(1-\alpha ,\lambda
,s)\right) \left\vert f^{\prime }(b)\right\vert \right. \\
&&+\left. \left( c_{4}(\alpha ,\lambda ,s)+c_{1}(1-\alpha ,\lambda
,s)\right) \left\vert f^{\prime }(a)\right\vert \right]
\end{eqnarray*}
\end{corollary}

\begin{remark}
In Theorem \ref{2.2}, if we take $s=1$, then we obtain the inequality (\ref%
{1-3}).
\end{remark}

\begin{remark}
In Theorem \ref{2.2}, if we take $\alpha =\frac{1}{2}$ and $\lambda =\frac{1%
}{3},$ then we have the following Simpson type inequality%
\begin{equation}
\left\vert \frac{1}{6}\left[ f(a)+4f\left( \frac{a+b}{2}\right) +f(b)\right]
-\frac{1}{b-a}\dint\limits_{a}^{b}f(x)dx\right\vert \leq \frac{b-a}{2}\left( 
\frac{5}{36}\right) ^{1-\frac{1}{q}}  \label{2-9}
\end{equation}%
\begin{eqnarray}
&&\times \left\{ \left( \frac{(2s+1)3^{s+1}+2}{3\times 6^{s+1}(s+1)(s+2)}%
\left\vert f^{\prime }(b)\right\vert ^{q}+\frac{2\times
5^{s+2}+(s-4)6^{s+1}-(2s+7)3^{s+1}}{3\times 6^{s+1}(s+1)(s+2)}\left\vert
f^{\prime }(a)\right\vert ^{q}\right) ^{\frac{1}{q}}\right.  \notag \\
&&\left. +\left( \frac{2\times 5^{s+2}+(s-4)6^{s+1}-(2s+7)3^{s+1}}{%
3.6^{s+1}(s+1)(s+2)}\left\vert f^{\prime }(b)\right\vert ^{q}+\frac{%
(2s+1)3^{s+1}+2}{3\times 6^{s+1}(s+1)(s+2)}\left\vert f^{\prime
}(a)\right\vert ^{q}\right) ^{\frac{1}{q}}\right\} ,  \notag
\end{eqnarray}%
which is the same of the inequality in \cite[Theorem 10]{SSO10} .
\end{remark}

\begin{remark}
In Theorem \ref{2.2} , if we take $\alpha =\frac{1}{2}$ and $\lambda =0,$
then we have following midpoint inequality%
\begin{eqnarray}
&&\left\vert f\left( \frac{a+b}{2}\right) -\frac{1}{b-a}\dint%
\limits_{a}^{b}f(x)dx\right\vert \leq \frac{b-a}{8}\left( \frac{2}{(s+1)(s+2)%
}\right) ^{\frac{1}{q}}  \label{2-10} \\
&&\times \left\{ \left( \frac{2^{1-s}\left( s+1\right) \left\vert f^{\prime
}(b)\right\vert ^{q}}{2}+\frac{2^{1-s}\left( 2^{s+2}-s-3\right) \left\vert
f^{\prime }(a)\right\vert ^{q}}{2}\right) ^{\frac{1}{q}}\right.  \notag \\
&&\left. +\left( \frac{2^{1-s}\left( s+1\right) \left\vert f^{\prime
}(a)\right\vert ^{q}}{2}+\frac{2^{1-s}\left( 2^{s+2}-s-3\right) \left\vert
f^{\prime }(b)\right\vert ^{q}}{2}\right) ^{\frac{1}{q}}\right\} .  \notag
\end{eqnarray}%
We note that the obtained midpoint inequality (\ref{2-10}) is better than
the inequality (\ref{1-4}). Because $\frac{s+1}{2}\leq 1$ and $\frac{%
2^{s+2}-s-3}{2}\leq \frac{2^{1-s}+1}{2^{1-s}}.$
\end{remark}

\begin{remark}
In Theorem \ref{2.2} , if we take $\alpha =\frac{1}{2}$ , and $\lambda =1,$
then we get the following trapezoid inequality%
\begin{eqnarray*}
&&\left\vert \frac{f\left( a\right) +f\left( b\right) }{2}-\frac{1}{b-a}%
\dint\limits_{a}^{b}f(x)dx\right\vert \leq \frac{b-a}{8}\left( \frac{2^{1-s}%
}{(s+1)(s+2)}\right) ^{\frac{1}{q}} \\
&&\times \left\{ \left( \left\vert f^{\prime }(b)\right\vert ^{q}+\left\vert
f^{\prime }(a)\right\vert ^{q}\left( 2^{s+1}+1\right) \right) ^{\frac{1}{q}%
}+\left( \left\vert f^{\prime }(a)\right\vert ^{q}+\left\vert f^{\prime
}(b)\right\vert ^{q}\left( 2^{s+1}+1\right) \right) ^{\frac{1}{q}}\right\}
\end{eqnarray*}
\end{remark}

Using Lemma \ref{2.1} we shall give another result for convex functions as
follows.

\begin{theorem}
\label{2.3}Let $f:I\subseteq \mathbb{R\rightarrow R}$ be a differentiable
mapping on $I^{\circ }$ such that $f^{\prime }\in L[a,b]$, where $a,b\in
I^{\circ }$ with $a<b$ and $\alpha ,\lambda \in \left[ 0,1\right] $. If $%
\left\vert f^{\prime }\right\vert ^{q}$ is $s$-convex on $[a,b]$, for some
fixed $s\in (0,1]$ and $q>1,$ then 
\begin{equation}
\left\vert I_{f}\left( \lambda ,\alpha ,a,b\right) \right\vert \leq \left(
b-a\right) \left( \frac{1}{p+1}\right) ^{\frac{1}{p}}\left( \frac{1}{s+1}%
\right) ^{\frac{1}{q}}  \label{2-12}
\end{equation}%
\begin{equation*}
\times \left\{ 
\begin{array}{cc}
\left[ \varepsilon _{1}^{1/p}(\alpha ,\lambda ,p)C_{f}^{1/q}(\alpha
,q)+\varepsilon _{1}^{1/p}(1-\alpha ,\lambda ,p)D_{f}^{1/q}(\alpha ,q)\right]
, & \alpha \lambda \leq 1-\alpha \leq 1-\lambda \left( 1-\alpha \right) \\ 
\left[ \varepsilon _{1}^{1/p}(\alpha ,\lambda ,p)C_{f}^{1/q}(\alpha
,q)+\varepsilon _{2}^{1/p}(1-\alpha ,\lambda ,p)D_{f}^{1/q}(\alpha ,q)\right]
, & \alpha \lambda \leq 1-\lambda \left( 1-\alpha \right) \leq 1-\alpha \\ 
\left[ \varepsilon _{2}^{1/p}(\alpha ,\lambda ,p)C_{f}^{1/q}(\alpha
,q)+\varepsilon _{1}^{1/p}(1-\alpha ,\lambda ,p)D_{f}^{1/q}(\alpha ,q)\right]
, & 1-\alpha \leq \alpha \lambda \leq 1-\lambda \left( 1-\alpha \right)%
\end{array}%
\right. ,
\end{equation*}%
where 
\begin{eqnarray}
C_{f}(\alpha ,q) &=&\left( 1-\alpha \right) \left[ \left\vert f^{\prime
}\left( \left( 1-\alpha \right) b+\alpha a\right) \right\vert
^{q}+\left\vert f^{\prime }\left( a\right) \right\vert ^{q}\right] ,
\label{2-12a} \\
\ D_{f}(\alpha ,q) &=&\alpha \left[ \left\vert f^{\prime }\left( \left(
1-\alpha \right) b+\alpha a\right) \right\vert ^{q}+\left\vert f^{\prime
}\left( b\right) \right\vert ^{q}\right] ,  \notag
\end{eqnarray}%
\begin{eqnarray}
\varepsilon _{1}(\alpha ,\lambda ,p) &=&\left( \alpha \lambda \right)
^{p+1}+\left( 1-\alpha -\alpha \lambda \right) ^{p+1},\   \label{2-12b} \\
\varepsilon _{2}(\alpha ,\lambda ,p) &=&\left( \alpha \lambda \right)
^{p+1}-\left( \alpha \lambda -1+\alpha \right) ^{p+1},  \notag
\end{eqnarray}%
and $\frac{1}{p}+\frac{1}{q}=1.$
\end{theorem}

\begin{proof}
From Lemma \ref{2.1} and by H\"{o}lder's integral inequality, we have%
\begin{eqnarray*}
&&\left\vert I_{f}\left( \lambda ,\alpha ,a,b\right) \right\vert \\
&\leq &\left( b-a\right) \left[ \dint\limits_{0}^{1-\alpha }\left\vert
t-\alpha \lambda \right\vert \left\vert f^{\prime }\left( tb+(1-t)a\right)
\right\vert dt+\dint\limits_{1-\alpha }^{1}\left\vert t-1+\lambda \left(
1-\alpha \right) \right\vert \left\vert f^{\prime }\left( tb+(1-t)a\right)
\right\vert dt\right] \\
&\leq &\left( b-a\right) \left\{ \left( \dint\limits_{0}^{1-\alpha
}\left\vert t-\alpha \lambda \right\vert ^{p}dt\right) ^{\frac{1}{p}}\left(
\dint\limits_{0}^{1-\alpha }\left\vert f^{\prime }\left( tb+(1-t)a\right)
\right\vert ^{q}dt\right) ^{\frac{1}{q}}\right.
\end{eqnarray*}%
\begin{equation}
+\left. \left( \dint\limits_{1-\alpha }^{1}\left\vert t-1+\lambda \left(
1-\alpha \right) \right\vert ^{p}dt\right) ^{\frac{1}{p}}\left(
\dint\limits_{1-\alpha }^{1}\left\vert f^{\prime }\left( tb+(1-t)a\right)
\right\vert ^{q}dt\right) ^{\frac{1}{q}}\right\} .  \label{2-13}
\end{equation}%
Since $\left\vert f^{\prime }\right\vert ^{q}$ is $s$-convex on $[a,b],$ for 
$\alpha \in \left[ 0,1\right) $ by the inequality (\ref{1-2}), we get 
\begin{eqnarray}
\dint\limits_{0}^{1-\alpha }\left\vert f^{\prime }\left( tb+(1-t)a\right)
\right\vert ^{q}dt &=&\left( 1-\alpha \right) \left[ \frac{1}{\left(
1-\alpha \right) \left( b-a\right) }\dint\limits_{a}^{\left( 1-\alpha
\right) b+\alpha a}\left\vert f^{\prime }\left( x\right) \right\vert ^{q}dx%
\right]  \notag \\
&\leq &\left( 1-\alpha \right) \left[ \frac{\left\vert f^{\prime }\left(
\left( 1-\alpha \right) b+\alpha a\right) \right\vert ^{q}+\left\vert
f^{\prime }\left( a\right) \right\vert ^{q}}{s+1}\right] .  \label{2-14}
\end{eqnarray}%
The inequality (\ref{2-14}) also holds for $\alpha =1$. Similarly, for $%
\alpha \in \left( 0,1\right] $ by the inequality (\ref{1-2}), we have 
\begin{eqnarray}
\dint\limits_{1-\alpha }^{1}\left\vert f^{\prime }\left( tb+(1-t)a\right)
\right\vert ^{q}dt &=&\alpha \left[ \frac{1}{\alpha \left( b-a\right) }%
\dint\limits_{\left( 1-\alpha \right) b+\alpha a}^{b}\left\vert f^{\prime
}\left( x\right) \right\vert ^{q}dx\right]  \notag \\
&\leq &\alpha \left[ \frac{\left\vert f^{\prime }\left( \left( 1-\alpha
\right) b+\alpha a\right) \right\vert ^{q}+\left\vert f^{\prime }\left(
b\right) \right\vert ^{q}}{s+1}\right] .  \label{2-15}
\end{eqnarray}%
The inequality (\ref{2-15}) also holds for $\alpha =0$. By simple computation%
\begin{equation}
\dint\limits_{0}^{1-\alpha }\left\vert t-\alpha \lambda \right\vert
^{p}dt=\left\{ 
\begin{array}{cc}
\frac{\left( \alpha \lambda \right) ^{p+1}+\left( 1-\alpha -\alpha \lambda
\right) ^{p+1}}{p+1}, & \alpha \lambda \leq 1-\alpha \\ 
\frac{\left( \alpha \lambda \right) ^{p+1}-\left( \alpha \lambda -1+\alpha
\right) ^{p+1}}{p+1}, & \alpha \lambda \geq 1-\alpha%
\end{array}%
\right. ,  \label{2-16}
\end{equation}%
and%
\begin{equation}
\dint\limits_{1-\alpha }^{1}\left\vert t-1+\lambda \left( 1-\alpha \right)
\right\vert ^{p}dt=\left\{ 
\begin{array}{cc}
\frac{\left[ \lambda \left( 1-\alpha \right) \right] ^{p+1}+\left[ \alpha
-\lambda \left( 1-\alpha \right) \right] ^{p+1}}{p+1}, & 1-\alpha \leq
1-\lambda \left( 1-\alpha \right) \\ 
\frac{\left[ \lambda \left( 1-\alpha \right) \right] ^{p+1}-\left[ \lambda
\left( 1-\alpha \right) -\alpha \right] ^{p+1}}{p+1}, & 1-\alpha \geq
1-\lambda \left( 1-\alpha \right)%
\end{array}%
\right. ,  \label{2-17}
\end{equation}%
thus, using (\ref{2-14})-(\ref{2-17}) in (\ref{2-13}), we obtain the
inequality (\ref{2-12}). This completes the proof.
\end{proof}

\begin{corollary}
Under the assumptions of Theorem \ref{2.3} with $s=1$, we have%
\begin{equation*}
\left\vert I_{f}\left( \lambda ,\alpha ,a,b\right) \right\vert \leq \left(
b-a\right) \left( \frac{1}{p+1}\right) ^{\frac{1}{p}}\left( \frac{1}{2}%
\right) ^{\frac{1}{q}}
\end{equation*}%
\begin{equation*}
\times \left\{ 
\begin{array}{cc}
\left[ \varepsilon _{1}^{1/p}(\alpha ,\lambda ,p)C_{f}^{1/q}(\alpha
,q)+\varepsilon _{1}^{1/p}(1-\alpha ,\lambda ,p)D_{f}^{1/q}(\alpha ,q)\right]
, & \alpha \lambda \leq 1-\alpha \leq 1-\lambda \left( 1-\alpha \right) \\ 
\left[ \varepsilon _{1}^{1/p}(\alpha ,\lambda ,p)C_{f}^{1/q}(\alpha
,q)+\varepsilon _{2}^{1/p}(1-\alpha ,\lambda ,p)D_{f}^{1/q}(\alpha ,q)\right]
, & \alpha \lambda \leq 1-\lambda \left( 1-\alpha \right) \leq 1-\alpha \\ 
\left[ \varepsilon _{2}^{1/p}(\alpha ,\lambda ,p)C_{f}^{1/q}(\alpha
,q)+\varepsilon _{1}^{1/p}(1-\alpha ,\lambda ,p)D_{f}^{1/q}(\alpha ,q)\right]
, & 1-\alpha \leq \alpha \lambda \leq 1-\lambda \left( 1-\alpha \right)%
\end{array}%
\right. ,
\end{equation*}%
where $\varepsilon _{1},\ \varepsilon _{2},\ C_{f}$ and $D_{f}$ are defined
as in (\ref{2-12a}).
\end{corollary}

\begin{remark}
In Theorem \ref{2.3}, if we take $\alpha =\frac{1}{2}$ and $\lambda =\frac{1%
}{3}$, then we have the following Simpson type inequality 
\begin{equation}
\left\vert \frac{1}{6}\left[ f(a)+4f\left( \frac{a+b}{2}\right) +f(b)\right]
-\frac{1}{b-a}\dint\limits_{a}^{b}f(x)dx\right\vert  \label{2-18}
\end{equation}%
\begin{equation*}
\leq \frac{b-a}{12}\left( \frac{1+2^{p+1}}{3\left( p+1\right) }\right) ^{%
\frac{1}{p}}\left\{ \left( \frac{\left\vert f^{\prime }\left( \frac{a+b}{2}%
\right) \right\vert ^{q}+\left\vert f^{\prime }\left( a\right) \right\vert
^{q}}{s+1}\right) ^{\frac{1}{q}}+\left( \frac{\left\vert f^{\prime }\left( 
\frac{a+b}{2}\right) \right\vert ^{q}+\left\vert f^{\prime }\left( b\right)
\right\vert ^{q}}{s+1}\right) ^{\frac{1}{q}}\right\} ,
\end{equation*}%
which is the same of the inequality (\ref{1-5}).
\end{remark}

\begin{remark}
In Theorem \ref{2.3}, if we take $\alpha =\frac{1}{2}$ and $\lambda =0,$
then we have the following midpoint inequality%
\begin{eqnarray*}
&&\left\vert f\left( \frac{a+b}{2}\right) -\frac{1}{b-a}\dint%
\limits_{a}^{b}f(x)dx\right\vert \\
&\leq &\frac{b-a}{4}\left( \frac{1}{p+1}\right) ^{\frac{1}{p}}\left\{ \left( 
\frac{\left\vert f^{\prime }\left( \frac{a+b}{2}\right) \right\vert
^{q}+\left\vert f^{\prime }\left( a\right) \right\vert ^{q}}{s+1}\right) ^{%
\frac{1}{q}}+\left( \frac{\left\vert f^{\prime }\left( \frac{a+b}{2}\right)
\right\vert ^{q}+\left\vert f^{\prime }\left( b\right) \right\vert ^{q}}{s+1}%
\right) ^{\frac{1}{q}}\right\} .
\end{eqnarray*}%
We note that by inequality 
\begin{equation*}
2^{s-1}\left\vert f^{\prime }\left( \frac{a+b}{2}\right) \right\vert
^{q}\leq \frac{\left\vert f^{\prime }\left( a\right) \right\vert
^{q}+\left\vert f^{\prime }\left( b\right) \right\vert ^{q}}{s+1}
\end{equation*}%
we have%
\begin{eqnarray*}
\left\vert f\left( \frac{a+b}{2}\right) -\frac{1}{b-a}\dint%
\limits_{a}^{b}f(x)dx\right\vert &\leq &\left( \frac{b-a}{4}\right) \left( 
\frac{1}{p+1}\right) ^{\frac{1}{p}}\left( \frac{1}{s+1}\right) ^{\frac{2}{q}}
\\
&&\times \left[ \left( \left( 2^{1-s}+s+1\right) \left\vert f^{\prime
}\left( a\right) \right\vert ^{q}+2^{1-s}\left\vert f^{\prime }\left(
b\right) \right\vert ^{q}\right) ^{\frac{1}{q}}\right. \\
&&+\left. \left( \left( 2^{1-s}+s+1\right) \left\vert f^{\prime }\left(
b\right) \right\vert ^{q}+2^{1-s}\left\vert f^{\prime }\left( a\right)
\right\vert ^{q}\right) ^{\frac{1}{q}}\right] ,
\end{eqnarray*}%
which is the same of the inequality (\ref{1-4a}).
\end{remark}

\begin{remark}
In Theorem \ref{2.3}, if we take $\alpha =\frac{1}{2}$ and $\lambda =1,$
then we have the following trapezoid inequality%
\begin{eqnarray}
&&\left\vert \frac{f\left( a\right) +f\left( b\right) }{2}-\frac{1}{b-a}%
\dint\limits_{a}^{b}f(x)dx\right\vert \leq \frac{b-a}{4}\left( \frac{1}{p+1}%
\right) ^{\frac{1}{p}}  \label{2-19} \\
&&\times \left\{ \left( \frac{\left\vert f^{\prime }\left( \frac{a+b}{2}%
\right) \right\vert ^{q}+\left\vert f^{\prime }\left( a\right) \right\vert
^{q}}{s+1}\right) ^{\frac{1}{q}}+\left( \frac{\left\vert f^{\prime }\left( 
\frac{a+b}{2}\right) \right\vert ^{q}+\left\vert f^{\prime }\left( b\right)
\right\vert ^{q}}{s+1}\right) ^{\frac{1}{q}}\right\} .  \notag
\end{eqnarray}%
We note that the obtained midpoint inequality (\ref{2-19}) is better than
the inequality (\ref{1-6}).
\end{remark}

\begin{theorem}
\label{2.5}Let $f:I\subseteq \mathbb{R\rightarrow R}$ be a differentiable
mapping on $I^{\circ }$ such that $f^{\prime }\in L[a,b]$, where $a,b\in
I^{\circ }$ with $a<b$ and $\alpha ,\lambda \in \left[ 0,1\right] $. If $%
\left\vert f^{\prime }\right\vert ^{q}$ is $s$-concave on $[a,b]$, for some
fixed $s\in (0,1]$ and $q>1,$ then the following inequality holds:%
\begin{equation}
\left\vert I_{f}\left( \lambda ,\alpha ,a,b\right) \right\vert \leq \left(
b-a\right) 2^{\frac{s-1}{q}}\left( \frac{1}{p+1}\right) ^{\frac{1}{p}}
\label{2-20}
\end{equation}%
\begin{equation*}
\times \left\{ 
\begin{array}{cc}
\left[ \varepsilon _{1}^{1/p}(\alpha ,\lambda ,p)E_{f}^{1/q}(\alpha
,q)+\varepsilon _{1}^{1/p}(1-\alpha ,\lambda ,p)F_{f}^{1/q}(\alpha ,q)\right]
, & \alpha \lambda \leq 1-\alpha \leq 1-\lambda \left( 1-\alpha \right) \\ 
\left[ \varepsilon _{1}^{1/p}(\alpha ,\lambda ,p)E_{f}^{1/q}(\alpha
,q)+\varepsilon _{2}^{1/p}(1-\alpha ,\lambda ,p)F_{f}^{1/q}(\alpha ,q)\right]
, & \alpha \lambda \leq 1-\lambda \left( 1-\alpha \right) \leq 1-\alpha \\ 
\left[ \varepsilon _{2}^{1/p}(\alpha ,\lambda ,p)E_{f}^{1/q}(\alpha
,q)+\varepsilon _{1}^{1/p}(1-\alpha ,\lambda ,p)F_{f}^{1/q}(\alpha ,q)\right]
, & 1-\alpha \leq \alpha \lambda \leq 1-\lambda \left( 1-\alpha \right)%
\end{array}%
\right. ,
\end{equation*}%
where%
\begin{equation*}
E_{f}(\alpha ,q)=\left( 1-\alpha \right) \left\vert f^{\prime }\left( \frac{%
\left( 1-\alpha \right) b+\left( 1+\alpha \right) a}{2}\right) \right\vert
^{q},\ F_{f}(\alpha ,q)=\alpha \left\vert f^{\prime }\left( \frac{\left(
2-\alpha \right) b+\alpha a}{2}\right) \right\vert ^{q},
\end{equation*}%
and $\varepsilon _{1},\ \varepsilon _{2}$ are defined as in (\ref{2-12a}).
\end{theorem}

\begin{proof}
We proceed similarly as in the proof Theorem \ref{2.3}. Since $\left\vert
f^{\prime }\right\vert ^{q}$ is $s-$concave on $[a,b],$ for $\alpha \in %
\left[ 0,1\right) $ by the inequality (\ref{1-2}), we get%
\begin{eqnarray}
\dint\limits_{0}^{1-\alpha }\left\vert f^{\prime }\left( tb+(1-t)a\right)
\right\vert ^{q}dt &=&\left( 1-\alpha \right) \left[ \frac{1}{\left(
1-\alpha \right) \left( b-a\right) }\dint\limits_{a}^{\left( 1-\alpha
\right) b+\alpha a}\left\vert f^{\prime }\left( x\right) \right\vert ^{q}dx%
\right]  \notag \\
&\leq &2^{s-1}\left( 1-\alpha \right) \left\vert f^{\prime }\left( \frac{%
\left( 1-\alpha \right) b+\left( 1+\alpha \right) a}{2}\right) \right\vert
^{q}  \label{2-21}
\end{eqnarray}%
The inequality (\ref{2-21}) also holds for $\alpha =1$. Similarly, for $%
\alpha \in \left( 0,1\right] $ by the inequality (\ref{1-2}), we have%
\begin{eqnarray}
\dint\limits_{1-\alpha }^{1}\left\vert f^{\prime }\left( tb+(1-t)a\right)
\right\vert ^{q}dt &=&\alpha \left[ \frac{1}{\alpha \left( b-a\right) }%
\dint\limits_{\left( 1-\alpha \right) b+\alpha a}^{b}\left\vert f^{\prime
}\left( x\right) \right\vert ^{q}dx\right]  \notag \\
&\leq &2^{s-1}\alpha \left\vert f^{\prime }\left( \frac{\left( 2-\alpha
\right) b+\alpha a}{2}\right) \right\vert ^{q}  \label{2-22}
\end{eqnarray}%
The inequality (\ref{2-22}) also holds for $\alpha =0$. Thus, using (\ref%
{2-16}),(\ref{2-17}),(\ref{2-21})and (\ref{2-22}) in (\ref{2-13}), we obtain
the inequality (\ref{2-20}). This completes the proof.
\end{proof}

\begin{corollary}
\label{2.7}Under the assumptions of Theorem \ref{2.5} with $s=1$, we have%
\begin{equation*}
\left\vert I_{f}\left( \lambda ,\alpha ,a,b\right) \right\vert \leq \left(
b-a\right) \left( \frac{1}{p+1}\right) ^{\frac{1}{p}}
\end{equation*}%
\begin{equation*}
\times \left\{ 
\begin{array}{cc}
\left[ \varepsilon _{1}^{1/p}(\alpha ,\lambda ,p)E_{f}^{1/q}(\alpha
,q)+\varepsilon _{1}^{1/p}(1-\alpha ,\lambda ,p)F_{f}^{1/q}(\alpha ,q)\right]
, & \alpha \lambda \leq 1-\alpha \leq 1-\lambda \left( 1-\alpha \right) \\ 
\left[ \varepsilon _{1}^{1/p}(\alpha ,\lambda ,p)E_{f}^{1/q}(\alpha
,q)+\varepsilon _{2}^{1/p}(1-\alpha ,\lambda ,p)F_{f}^{1/q}(\alpha ,q)\right]
, & \alpha \lambda \leq 1-\lambda \left( 1-\alpha \right) \leq 1-\alpha \\ 
\left[ \varepsilon _{2}^{1/p}(\alpha ,\lambda ,p)E_{f}^{1/q}(\alpha
,q)+\varepsilon _{1}^{1/p}(1-\alpha ,\lambda ,p)F_{f}^{1/q}(\alpha ,q)\right]
, & 1-\alpha \leq \alpha \lambda \leq 1-\lambda \left( 1-\alpha \right)%
\end{array}%
\right. ,
\end{equation*}%
where $\varepsilon _{1},\ \varepsilon _{2},\ E_{f}$ and $F_{f}$ are defined
as in Theorem \ref{2.5}.
\end{corollary}

\begin{remark}
In Theorem \ref{2.5}, if we take $\alpha =\frac{1}{2}$ and $\lambda =1,$
then we have the following trapezoid inequality%
\begin{eqnarray*}
&&\left\vert \frac{f\left( a\right) +f\left( b\right) }{2}-\frac{1}{b-a}%
\dint\limits_{a}^{b}f(x)dx\right\vert \\
&\leq &\frac{b-a}{4}\left( \frac{1}{p+1}\right) ^{\frac{1}{p}}\times \left( 
\frac{1}{2}\right) ^{\frac{1-s}{q}}\left[ \left\vert f^{\prime }\left( \frac{%
3b+a}{4}\right) \right\vert +\left\vert f^{\prime }\left( \frac{3a+b}{4}%
\right) \right\vert \right]
\end{eqnarray*}%
which is the same of the inequality in \cite[Theorem 8 (i)]{P10}.
\end{remark}

\begin{remark}
In Theorem \ref{2.5}, if we take $\alpha =\frac{1}{2}$ and $\lambda =0,$
then we have the following midpoint inequality%
\begin{eqnarray*}
&&\left\vert f\left( \frac{a+b}{2}\right) -\frac{1}{b-a}\dint%
\limits_{a}^{b}f(x)dx\right\vert \\
&\leq &\frac{b-a}{4}\left( \frac{1}{p+1}\right) ^{\frac{1}{p}}\times \left( 
\frac{1}{2}\right) ^{\frac{1-s}{q}}\left[ \left\vert f^{\prime }\left( \frac{%
3b+a}{4}\right) \right\vert +\left\vert f^{\prime }\left( \frac{3a+b}{4}%
\right) \right\vert \right]
\end{eqnarray*}%
which is the same of the inequality in \cite[Theorem 8 (ii)]{P10}.
\end{remark}

\begin{remark}
In Theorem \ref{2.5}, if we take $\alpha =\frac{1}{2}$ and $\lambda =1,$
then we have the following trapezoid inequality%
\begin{eqnarray}
&&\left\vert \frac{f\left( a\right) +f\left( b\right) }{2}-\frac{1}{b-a}%
\dint\limits_{a}^{b}f(x)dx\right\vert  \label{2-23} \\
&\leq &\frac{b-a}{4}\left( \frac{1}{p+1}\right) ^{\frac{1}{p}}\left[
\left\vert f^{\prime }\left( \frac{3b+a}{4}\right) \right\vert +\left\vert
f^{\prime }\left( \frac{3a+b}{4}\right) \right\vert \right]  \notag
\end{eqnarray}%
which is the same of the inequality in \cite[Theorem 2]{KBOP07}.
\end{remark}

\begin{remark}
In Theorem \ref{2.5}, if we take $\alpha =\frac{1}{2}$ and $\lambda =0,$
then we have the following trapezoid inequality%
\begin{eqnarray}
&&\left\vert f\left( \frac{a+b}{2}\right) -\frac{1}{b-a}\dint%
\limits_{a}^{b}f(x)dx\right\vert  \label{2-24} \\
&\leq &\frac{b-a}{4}\left( \frac{1}{p+1}\right) ^{\frac{1}{p}}\left[
\left\vert f^{\prime }\left( \frac{3b+a}{4}\right) \right\vert +\left\vert
f^{\prime }\left( \frac{3a+b}{4}\right) \right\vert \right]  \notag
\end{eqnarray}%
which is the same of the inequality in \cite[Theorem 2.5]{ADK11}.
\end{remark}

\begin{remark}
In Theorem \ref{2.5}, since $\left\vert f^{\prime }\right\vert ^{q},\ q>1,$
is concave on $\left[ a,b\right] ,$ using the power mean inequality, we have%
\begin{eqnarray*}
\left\vert f^{\prime }\left( \lambda x+\left( 1-\lambda \right) y\right)
\right\vert ^{q} &\geq &\lambda \left\vert f^{\prime }\left( x\right)
\right\vert ^{q}+\left( 1-\lambda \right) \left\vert f^{\prime }\left(
y\right) \right\vert ^{q} \\
&\geq &\left( \lambda \left\vert f^{\prime }\left( x\right) \right\vert
+\left( 1-\lambda \right) \left\vert f^{\prime }\left( y\right) \right\vert
\right) ^{q},
\end{eqnarray*}%
$\forall x,y\in \left[ a,b\right] $ and $\lambda \in \left[ 0,1\right] .$
Hence%
\begin{equation*}
\left\vert f^{\prime }\left( \lambda x+\left( 1-\lambda \right) y\right)
\right\vert \geq \lambda \left\vert f^{\prime }\left( x\right) \right\vert
+\left( 1-\lambda \right) \left\vert f^{\prime }\left( y\right) \right\vert
\end{equation*}%
so $\left\vert f^{\prime }\right\vert $ is also concave. Then by the
inequality (\ref{1-1}), we have%
\begin{equation}
\left\vert f^{\prime }\left( \frac{3b+a}{4}\right) \right\vert +\left\vert
f^{\prime }\left( \frac{3a+b}{4}\right) \right\vert \leq 2\left\vert
f^{\prime }\left( \frac{a+b}{2}\right) \right\vert .  \label{2-25}
\end{equation}%
Thus, using the inequality (\ref{2-25}) in (\ref{2-23}) and (\ref{2-24}) we
get%
\begin{eqnarray*}
\left\vert \frac{f\left( a\right) +f\left( b\right) }{2}-\frac{1}{b-a}%
\dint\limits_{a}^{b}f(x)dx\right\vert &\leq &\frac{b-a}{2}\left( \frac{1}{p+1%
}\right) ^{\frac{1}{p}}\left\vert f^{\prime }\left( \frac{a+b}{2}\right)
\right\vert , \\
\left\vert f\left( \frac{a+b}{2}\right) -\frac{1}{b-a}\dint%
\limits_{a}^{b}f(x)dx\right\vert &\leq &\frac{b-a}{2}\left( \frac{1}{p+1}%
\right) ^{\frac{1}{p}}\left\vert f^{\prime }\left( \frac{a+b}{2}\right)
\right\vert .
\end{eqnarray*}
\end{remark}

\section{Some applications for special means}

Let us recall the following special means of arbitrary real numbers $a,b$
with $a\neq b$ and $\alpha \in \left[ 0,1\right] :$

\begin{enumerate}
\item The weighted arithmetic mean%
\begin{equation*}
A_{\alpha }\left( a,b\right) :=\alpha a+(1-\alpha )b,~a,b\in 
\mathbb{R}
.
\end{equation*}

\item The unweighted arithmetic mean%
\begin{equation*}
A\left( a,b\right) :=\frac{a+b}{2},~a,b\in 
\mathbb{R}
.
\end{equation*}

\item Then $p-$Logarithmic mean%
\begin{equation*}
L_{p}\left( a,b\right) :=\ \left( \frac{b^{p+1}-a^{p+1}}{(p+1)(b-a)}\right)
^{\frac{1}{p}}\ ,\ p\in 
\mathbb{R}
\backslash \left\{ -1,0\right\} ,\ a,b>0.
\end{equation*}
\end{enumerate}

From known Example 1 in \cite{HM94}, we may find that for any \ $s\in \left(
0,1\right) $ and $\beta >0,$ $f:\left[ 0,\infty \right) \rightarrow \left[
0,\infty \right) $, $f(t)=\beta t^{s},\ f\in K_{s}^{2}.$

Now, using the resuls of Section 2, some new inequalities are derived for
the above means.

\begin{proposition}
Let $a,b\in 
\mathbb{R}
$ with $0<a<b,\ q\geq 1$ and \ $s\in \left( 0,\frac{1}{q}\right) $. Then
\end{proposition}

\begin{theorem}
(i) for $\alpha \lambda \leq 1-\alpha \leq 1-\lambda \left( 1-\alpha \right) 
$ we have%
\begin{eqnarray*}
&&\left\vert \lambda A_{\alpha }\left( a^{s+1},b^{s+1}\right) +\left(
1-\lambda \right) A_{\alpha }^{s+1}\left( a,b\right) -L_{s+1}^{s+1}\left(
a,b\right) \right\vert \\
&\leq &\left( b-a\right) \left( s+1\right) \left[ \gamma _{2}^{1-\frac{1}{q}%
}(\alpha ,\lambda )\left( c_{1}(\alpha ,\lambda ,s)b^{sq}+c_{2}(\alpha
,\lambda ,s)a^{sq}\right) ^{\frac{1}{q}}\right. \\
&&+\left. \gamma _{2}^{1-\frac{1}{q}}(1-\alpha ,\lambda )\left(
c_{2}(1-\alpha ,\lambda ,s)b^{sq}+c_{1}(1-\alpha ,\lambda ,s)a^{sq}\right) ^{%
\frac{1}{q}}\right] ,
\end{eqnarray*}

(ii) for $\alpha \lambda \leq 1-\lambda \left( 1-\alpha \right) \leq
1-\alpha $ we have%
\begin{eqnarray*}
&&\left\vert \lambda A_{\alpha }\left( a^{s+1},b^{s+1}\right) +\left(
1-\lambda \right) A_{\alpha }^{s+1}\left( a,b\right) -L_{s+1}^{s+1}\left(
a,b\right) \right\vert \\
&\leq &\left( b-a\right) \left( s+1\right) \left[ \gamma _{2}^{1-\frac{1}{q}%
}(\alpha ,\lambda )\left( c_{1}(\alpha ,\lambda ,s)b^{sq}+c_{2}(\alpha
,\lambda ,s)a^{sq}\right) ^{\frac{1}{q}}\right. \\
&&+\left. \gamma _{1}^{1-\frac{1}{q}}(1-\alpha ,\lambda )\left(
c_{4}(1-\alpha ,\lambda ,s)b^{sq}+c_{3}(1-\alpha ,\lambda ,s)a^{sq}\right) ^{%
\frac{1}{q}}\right] ,
\end{eqnarray*}

(iii) for $1-\alpha \leq \alpha \lambda \leq 1-\lambda \left( 1-\alpha
\right) $ we have%
\begin{eqnarray*}
&&\left\vert \lambda A_{\alpha }\left( a^{s+1},b^{s+1}\right) +\left(
1-\lambda \right) A_{\alpha }^{s+1}\left( a,b\right) -L_{s+1}^{s+1}\left(
a,b\right) \right\vert \\
&\leq &\left( b-a\right) \left( s+1\right) \left[ \gamma _{1}^{1-\frac{1}{q}%
}(\alpha ,\lambda )\left( c_{3}(\alpha ,\lambda ,s)b^{sq}+c_{4}(\alpha
,\lambda ,s)a^{sq}\right) ^{\frac{1}{q}}\right. \\
&&+\left. \gamma _{2}^{1-\frac{1}{q}}(1-\alpha ,\lambda )\left(
c_{2}(1-\alpha ,\lambda ,s)b^{sq}+c_{1}(1-\alpha ,\lambda ,s)a^{sq}\right) ^{%
\frac{1}{q}}\right]
\end{eqnarray*}%
where $\gamma _{1},\ \gamma _{2},\ c_{1},\ \ c_{2},\ c_{3},\ c_{4}$ numbers
are defined as in Theorem \ref{2.2}.
\end{theorem}

\begin{proof}
The assertion follows from applied the inequalities in Theorem \ref{2.2} to
the function $f(t)=t^{s+1}$,$\ t\in \left[ a,b\right] $ and $s\in \left( 0,%
\frac{1}{q}\right) ,$ which implies that $f^{\prime }(t)=(s+1)t^{s}$,$\ \
t\in \left[ a,b\right] $ and $\left\vert f^{\prime }(t)\right\vert
^{q}=(s+1)^{q}t^{qs}$,$\ \ t\in \left[ a,b\right] $ is a $s$-convex function
in the second sense since $qs\in \left( 0,1\right) $ and $(s+1)^{q}>0.$
\end{proof}

\begin{proposition}
Let $a,b\in 
\mathbb{R}
$ with $0<a<b,\ p,q>1,\ \frac{1}{p}+\frac{1}{q}=1$ and \ $s\in \left( 0,%
\frac{1}{q}\right) $we have the following inequality:%
\begin{equation*}
\left\vert \lambda A_{\alpha }\left( a^{s+1},b^{s+1}\right) +\left(
1-\lambda \right) A_{\alpha }^{s+1}\left( a,b\right) -L_{s+1}^{s+1}\left(
a,b\right) \right\vert \leq \left( b-a\right) \left( \frac{1}{p+1}\right) ^{%
\frac{1}{p}}\left( s+1\right) ^{1-\frac{1}{q}}
\end{equation*}%
\begin{equation*}
\times \left\{ 
\begin{array}{cc}
\left[ \varepsilon _{1}^{1/p}(\alpha ,\lambda ,p)C_{s}^{1/q}(\alpha
,q)+\varepsilon _{1}^{1/p}(1-\alpha ,\lambda ,p)D_{s}^{1/q}(\alpha ,q)\right]
, & \alpha \lambda \leq 1-\alpha \leq 1-\lambda \left( 1-\alpha \right) \\ 
\left[ \varepsilon _{1}^{1/p}(\alpha ,\lambda ,p)C_{s}^{1/q}(\alpha
,q)+\varepsilon _{2}^{1/p}(1-\alpha ,\lambda ,p)D_{s}^{1/q}(\alpha ,q)\right]
, & \alpha \lambda \leq 1-\lambda \left( 1-\alpha \right) \leq 1-\alpha \\ 
\left[ \varepsilon _{2}^{1/p}(\alpha ,\lambda ,p)C_{s}^{1/q}(\alpha
,q)+\varepsilon _{1}^{1/p}(1-\alpha ,\lambda ,p)D_{s}^{1/q}(\alpha ,q)\right]
, & 1-\alpha \leq \alpha \lambda \leq 1-\lambda \left( 1-\alpha \right)%
\end{array}%
\right. ,
\end{equation*}%
where 
\begin{equation*}
C_{s}(\alpha ,q)=\left( 1-\alpha \right) \left[ A_{\alpha }^{sq}\left(
a,b\right) +a^{sq}\right] ,\ D_{s}(\alpha ,q)=\alpha \left[ A_{\alpha
}^{sq}\left( a,b\right) +b^{sq}\right] ,
\end{equation*}%
and $\varepsilon _{1}$ and$\ \varepsilon _{2}$ numbers are defined as in (%
\ref{2-12b}).
\end{proposition}

\begin{proof}
The assertion follows from applied the inequality (\ref{2-12}) to the
function $f(t)=t^{s+1}$,$\ t\in \left[ a,b\right] $ and $s\in \left( 0,\frac{%
1}{q}\right) ,$ which implies that $f^{\prime }(t)=(s+1)t^{s}$,$\ \ t\in %
\left[ a,b\right] $ and $\left\vert f^{\prime }(t)\right\vert
^{q}=(s+1)^{q}t^{qs}$,$\ \ t\in \left[ a,b\right] $ is a $s$-convex function
in the second sense since $qs\in \left( 0,1\right) $ and $(s+1)^{q}>0.$
\end{proof}

\end{document}